\documentclass{article}

\usepackage[english]{babel}
\usepackage[T1]{fontenc}
\usepackage[utf8x]{inputenc}

\usepackage{amsthm}
\usepackage{mathrsfs}
\usepackage{amsfonts}
\usepackage{amssymb}
\usepackage{amsmath}
\usepackage{thmtools, thm-restate}

\usepackage{color}
\usepackage{graphicx}
\usepackage{graphics}
\usepackage{xspace}
\usepackage[linesnumbered,ruled,vlined]{algorithm2e}

\def\cE{\mathcal{E}}
\def\cF{\mathcal{F}}
\def\cG{\mathcal{G}}
\def\cH{\mathcal{H}}

\def\cJ{\mathcal{J}}

\def\trace{\ensuremath{traces}}
\def\tracek{\ensuremath{\trace_k}}

\def\Tenum{\textsc{Trans-Enum}\xspace}
\def\Thyp{\textsc{Trans-Hyp}\xspace}
\def\VCdim{\textsc{VC-dim}\xspace}
\def\ext{ext}
\def\extk{\ext_{k}}

\newtheorem{thm}{Theorem}
\newtheorem{theorem}[thm]{Theorem}

\newtheorem{lemma}[thm]{Lemma}

\newtheorem{corollary}[thm]{Corollary}

\newtheorem{proposition}[thm]{Proposition}

\usepackage{tabularx, environ}

\title{Enumeration of minimal transversals of hypergraphs of bounded VC-dimension}
\author{A. Mary}

\begin{document}
\maketitle
\begin{abstract}
	We consider the problem of enumerating all minimal transversals (also called minimal hitting sets) of a hypergraph $\cH$. An equivalent formulation of this problem known as the \emph{transversal hypergraph} problem (or \emph{hypergraph dualization} problem) is to decide, given two hypergraphs, whether one corresponds to the set of minimal transversals of the other. The existence of a polynomial time algorithm to solve this problem is a long standing open question. In \cite{fredman_complexity_1996}, the authors present the first sub-exponential algorithm to solve the transversal hypergraph problem which runs in quasi-polynomial time, making it unlikely that the problem is (co)NP-complete.

	In this paper,  we show that when one of the two hypergraphs is of bounded VC-dimension, the transversal hypergraph problem can be solved in polynomial time, or equivalently that if $\cH$ is a hypergraph of bounded VC-dimension, then there exists an incremental polynomial time algorithm to enumerate its minimal transversals.
	This result generalizes most of the previously known polynomial cases in the literature since they almost all consider classes of hypergraphs of bounded VC-dimension.
	As a consequence, the hypergraph transversal problem is solvable in polynomial time for any class of hypergraphs closed under partial subhypergraphs.
	We also show that the proposed algorithm runs in quasi-polynomial time in general hypergraphs and runs in polynomial time if the conformality of the hypergraph is bounded, which is one of the few known polynomial cases where the VC-dimension is unbounded.
\end{abstract}

\section{Introduction}
A hypergraph $\cH$ is a couple $(V(\cH),\cE(\cH))$ where $V(\cH)$ is a finite set called the \emph{vertices} and  $\cE(\cH) \subseteq 2^{V(\cH)}$ is a family of subsets of $V(\cH)$ called the \emph{hyperedges} of the hypergraph. By abuse of notation, we often treat a hypergraph as its set of hyperedges when the set of vertices is clear from the context, and we say that a set of vertices $X$ belongs to $\cH$ if $X$ is a hyperedge of $\cH$.
A set $U\subseteq V(\cH)$ is called a \emph{transversal} (or a hitting set) of $\cH$ if $U\cap F\neq \emptyset$ for all $F\in \cE(\cH)$. A transversal $U$ is said to be \emph{minimal} if it does not contain any other transversal.
The set of minimal transversals of $\cH$ forms another hypergraph on the same vertex set denoted by $Tr(\cH)$ and referred to as the \emph{transversal hypergraph} or the \emph{dual hypergraph} of $\cH$ \cite{berge_hypergraphs_1989}. In this paper we are interested in the problem of finding $Tr(\cH)$ given $\cH$.\\

\noindent\textbf{\textsc{Trans-Enum}}\\
\noindent \textbf{Input} : A hypergraph $\cH$. \\
\noindent \textbf{Output} : All minimal transversals of $\cH$, i.e. $Tr(\cH)$.\\

Since $Tr(\cH)$ could be exponentially larger than $\cH$, this problem falls into the category of enumeration problems. To measure the complexity of algorithms that solve this kind of problems, we usually take into account both the input size (the size of $\cH$) and the output size (the size of $Tr(\cH)$). With this paradigm (called the \emph{output-sensitive} approach), an algorithm is said to be \emph{output-polynomial} if its running time is a polynomial in $|\cH|$ and $|Tr(\cH)|$. We say that an algorithm runs in \emph{incremental polynomial time} if it can find $\ell$ minimal transversals in time polynomial in $|\cH|$ and $\ell$. The problem admits an incremental polynomial time algorithm if and only if the following problem can be solved in polynomial time (in a classical sense):\\

\noindent\textbf{\textsc{Trans-Hyp}}\\
\noindent \textbf{Input} : Two hypergraphs $\cH$ and $\cG$ on the same vertex set $V$ with $\cG \subseteq Tr(\cH)$. \\
\noindent \textbf{Output} : Either answer that $\cG = Tr(\cH)$ or find $T\in Tr(\cH) \setminus \cG$ \\

We assume throughout the paper that $\cH$ is \emph{Sperner} i.e. no hyperedge of $\cH$ contains another hyperedge. This assumption can be made without loss of generality since otherwise $\cH$ has the same minimal transversals as the restriction to its inclusion-wise minimal hyperedges. For Sperner hypergraphs, it is well known that $\cH$ and $Tr(\cH)$ form a duality relationship in the sense that $Tr(\cH)=\cG$ if and only if $Tr(\cG) = \cH$ \cite{berge_hypergraphs_1989}. The \textsc{Trans-Hyp} problem corresponds to deciding whether two hypergraphs are dual, and to finding a counter-example otherwise. It has been shown in \cite{bioch_complexity_1995} that the simple decision version of this problem (without requiring a counter-example) is equivalent to \textsc{Trans-Hyp}.

The \textsc{Trans-Hyp} problem has been extensively studied due to its equivalence to many other important problems (see e.g. \cite{bioch_complexity_1995,
	goos_hypergraph_2002,eiter_computational_2008,fiat_output-sensitive_2009,
	gunopulos_data_1997,
	khachiyan_dualization_2007,elbassioni_global_2019,takata_worst-case_2008}). It is a long standing open question to decide whether \textsc{Trans-Hyp} can be solved in polynomial time. The best known algorithms to solve it run in quasi-polynomial time $nN^{o(\log N)}$ where $N=|\cH| + |\cG|$ and $n=|V(\cH)|$ \cite{fredman_complexity_1996}. The problem is then very unlikely to be NP-hard.
The VC-dimension of a hypergraph was introduced in \cite{vapnik_uniform_1971} and has been shown to be an important parameter for many different applications.
Given a class of hypergraph $\mathscr{H}$, the \Thyp problem restricted to $\mathscr{H}$ consists in all instances $(\cH, \cG)$ of \Thyp such that at least one of $\cH$ or $\cG$ belongs to $\mathscr{H}$.
The main result of this paper is that \Thyp can be solved in polynomial time in hypergraphs of bounded VC-dimension, or equivalently, that \Tenum can be solved in incremental polynomial time for hypergraphs of bounded VC-dimension.  This answers an open question proposed in \cite{10.1145/3173127.3173138}.

\begin{restatable}{Theorem}{mainthm}
	\label{thm:main}
	Let $\cH$ be a hypergraph with $\VCdim(\cH)<k$ and let $\cG \subseteq Tr(\cH)$. One can decide in time $O(2^k(n|\cG|)^{k+1} + n^{2k+1}|\cG|) $ whether $\cG = Tr(\cH)$ and find $T\in Tr(\cH) \setminus  \cG$ otherwise.
\end{restatable}

Many different polynomial time algorithms for $\Thyp$ have been designed for special classes of hypergraphs in the literature.
While more efficient algorithms are generally presented, Theorem \ref{thm:main} generalizes many of those results since we observe that considered classes often  have bounded VC-dimension. Among the already known polynomial cases directly covered by Theorem \ref{thm:main} we can cite:
\begin{itemize}
	\item  bounded hyperedge size, bounded edge-intersections, $\beta$-acyclic \cite{eiter_identifying_1995,goos_generating_2004,khachiyan_dualization_2007,boros_efficient_2000}
	\item $\delta$-sparse hypergraphs, bounded degree,  bounded tree-width, totally unimodular hypergraphs, balanced hypergraphs \cite{khachiyan_computing_2007,mishra_generating_1997,eiter_new_nodate,doi:10.1137/18M1198995}
	\item Several geometrically defined hypergraphs: axis-parallel hyper-rectangles, half-spaces, axis-parallel hyperplanes, balls, polytopes with fixed number of facets/vertices \cite{fiat_output-sensitive_2009,elbassioni_global_2019}
\end{itemize}

We observe that all those classes are closed under partial subhypergraphs. A partial subhypergraph of a hypergraph $\cH$ is a hypergraph obtained from $\cH$ by selecting a subset of hyperedges $\cE'\subseteq \cE(\cH)$, a subset of vertices $V'\subseteq V(\cH)$ and considering the hyperedges of $\cE'$ restricted to $V'$, i.e. the hypergraph $(V',\{F\cap V' \mid F\in \cE'\}$. As an important corollary of Theorem \ref{thm:main} we obtain the following.
\begin{restatable}{corollary}{closedHyper}\label{cor:closed}
	\Thyp can be solved in polynomial time in any proper class of hypergraphs closed under partial subhypergraph.
\end{restatable}

One of the most general classes for which we already know that \Thyp is solvable in polynomial time and which is not covered by Theorem \ref{thm:main} is the class of $k$-conformal hypergraphs \cite{khachiyan_dualization_2007}. While this class is not closed under partial subhypergraph and does not have a bounded VC-dimension we show that the algorithm developed in this paper runs also in polynomial time if the hypergraph is $k$-conformal.
To the best of our knowledge, the two main cases for which the polynomiality cannot be directly deduced from the results present in this paper are the class of $k$-degenerate hypergraphs \cite{eiter_new_nodate} and the class of $r$-exact hypergraphs \cite{elbassioni_polynomial-time_2010}.

\paragraph{Organization of the paper.}
Section~\ref{sec:prelim} introduces the notions of traces, $k$-compatibility and VC-dimension that we use throughout the paper. In Section~\ref{sec:mainalgo}, we present our main algorithm for \Thyp when one input hypergraph has bounded VC-dimension, and we prove its correctness, running time and some consequences. Finally, Section~\ref{sec:closed} studies hypergraph classes closed under subhypergraph relations and shows how our result implies polynomiality for classes closed under partial subhypergraphs.

\section{Preliminaries}\label{sec:prelim}

A trace on $V:=V(\cH)$ is a pair $(T,S)$ with $S\subseteq V$ and $T \subseteq S$. The size of the trace $(T,S)$ is defined as $|S|$, and it is called a $k$-trace if $|S|=k$. We say that a subset $F\subseteq V$ \emph{realizes} a trace $(T,S)$ if $F\cap S = T$, and we denote by
\[
\trace_k(F,V) := \{(F\cap S, S) \mid S\subseteq V,\ |S|=k\}
\]
the set of $k$-traces realized by $F$ on $V$.
When the set $V$ is clear from the context, we will simply use $\trace_k(F)$ instead of $\trace_k(F,V)$.

For a hypergraph $\cH$  we denote by 
\[\trace_k(\cH):= \bigcup\limits_{F\in\cH} \trace_k(F,V(\cH))\] 
the set of traces realized by its hyperedges. In other words, a $k$-trace $(T,S)$ belongs to $\trace_k(\cH)$ if there exists a hyperedge $F\in \cH$ such that $F\cap S = T$.
A subset $E\subseteq V(\cH)$ is \emph{$k$-compatible} with $\cH$ if $\trace_k(E)\subseteq \trace_k(\cH)$, i.e., for each $k$-subset $S\subseteq V(\cH)$ there exists $F\in \cH$ such that $E\cap S = F\cap S$. Given a hypergraph $\cH$ and $k\leq |V(\cH)|$, the $k$-extension of $\cH$, $\extk(\cH)$, is the hypergraph on $V(\cH)$ whose hyperedges are all subsets of $V$ that are $k$-compatible with $\cH$, i.e. $\extk(\cH):=(V(\cH), \{E \mid E \subseteq V, \trace_k(E) \subseteq \trace_k(\cH) \})$. Since in particular every hyperedge of $\cH$ is $k$-compatible with $\cH$, $\cH \subseteq \extk(\cH)$.


A subset of vertices $U \subseteq V(\cH)$ is \emph{shattered} in $\cH$ if for all $U'\subseteq U$, there exists $F\in \cH$ such that $F\cap U=U'$ i.e. if $(U',U)\in \trace_{|U|}(\cH)$ for all $U'\subseteq U$. The \emph{VC-dimension} of $\cH$, $\VCdim(\cH)$, is the size of its largest shattered set. Using the trace definition, $\VCdim(\cH) < k$ if for all $k$-subsets $S$ of $V$ there exists $T\subseteq S $  such that $(T,S)\notin \tracek(\cH)$.
One of the most important properties of hypergraphs of bounded dimension is given by the Sauer–Shelah Lemma.

\begin{lemma}[Sauer–Shelah Lemma]\label{lem:sauer}
	If $\cH$ is of VC-dimension strictly less than $k$, then $|\cH|=O(|V|^k)$.
\end{lemma}

\begin{corollary}\label{cor:sauer}
	If $\cH$ is of VC-dimension strictly less than $k$, then $|\extk(\cH)|=O(|V|^k)$.
\end{corollary}

\begin{proof}
	Since by definition, $\extk(\cH)$ has exactly the same $k$-traces as $\cH$, a subset of vertices of size $k$ is shattered in $\extk(\cH)$ if and only if it is shattered in $\cH$. Since $\cH$ shatters no subset of size $k$, $\VCdim(\extk(\cH))<k$.  The result follows from Lemma \ref{lem:sauer}.
\end{proof}

Given a hypergraph $\cH$ and a subset of hyperedges $\cE'\subseteq  \cH$, the hypergraph $( \bigcup\limits_{F \in \cE'} F ,\cE')$ is called the \emph{partial hypergraph} of $\cH$ induced by $\cE'$.  Given a subset of vertices $V'\subseteq V(\cH)$ the hypergraph $(V',\{F\cap V' \mid F\in \cH, \; F\cap V'\neq \emptyset \})$ is the \emph{subhypergraph} induced by $V'$. Given $V'\subseteq V$ and $\cE'\subseteq \cH$ the \emph{partial subhypergraph} induced by $V'$ and $\cE'$ is the hypergraph $(V',\{F\cap V' \mid F\in \cE', \; F\cap V'\neq \emptyset \})$. A partial subhypergraph of $\cH$ that does not contain isolated vertices (i.e. vertices that do not belong to any hyperedge) is a partial hypergraph of a subhypergraph of $\cH$ or equivalently a subhypergraph of a partial hypergraph of $\cH$.

\section{Main algorithm}\label{sec:mainalgo}

In this section we assume that we are given a hypergraph $\cH$ with  $\VCdim(\cH)<k$ and a hypergraph $\cG\subseteq Tr(\cH)$. We want to decide whether $Tr(\cH) = \cG$, or equivalently whether $\cG=Tr(\cH)$, and find a new minimal transversal in $Tr(\cH) \setminus \cG$ otherwise. We present Algorithm \ref{algo:1} to solve this problem whose running time is $O(2^k(n|\cG|)^{k+1} + n^{2k+1}|\cG|)$ as stated by Theorem \ref{thm:main}.

Notice that we can assume that $\cH\subseteq Tr(\cG)$ since otherwise the answer is no, and we can easily find a new minimal transversal in $Tr(\cH) \setminus \cG$. Indeed assume that there exists $E\in \cF$ such that $E$ is not a minimal transversal of $\cG$, i.e. there exists $x\in E$ such that $E\setminus \{x\}$ is also a transversal of $\cG$, then any minimal transversal $T$ of $\cH$ included in $(V\setminus E)\cup \{x\}$ belongs to $Tr(\cH)\setminus \cG$.


The algorithm will actually either answer that $Tr(\cG)=\cH$ or find a new minimal transversal  $T\in Tr(\cG) \setminus \cH$ otherwise. In the latter case, given $T$, one can easily find a new minimal transversal of $ Tr(\cH) \setminus \cG$. Indeed it is enough to output any minimal transversal of $\cH$ included in $V\setminus T$.

We are trying to find $T\in Tr(\cG)$ such that $T\notin \cH$. The main strategy is based on the observation that if such a $T$ exists, then either $T$ is $k$-compatible with $\cH$ (i.e. $T\in \extk(\cH)$) or it realizes a $k$-trace that is not in $\tracek(\cH)$. It is easy to check whether there exists a $T$ that satisfies the latter case in general hypergraphs, and we show that the former case can be checked in polynomial time when $\VCdim(\cH)<k$.
The strategy is summarized in Algorithm \ref{algo:1}.

\begin{algorithm}\label{algo:1}
	\SetKwInOut{Input}{input}\SetKwInOut{Output}{output}

	\Input{Two hypergraphs $\cH$ and $\cG$ and $k\in \mathbb{N}$}
	\Output{Yes if $Tr(\cG)=\cH$ or $E\in Tr(\cG)\setminus \cH$ otherwise}
	\BlankLine
	\Begin{


		\ForEach{$k$-trace $(T,S)\notin \tracek(\cH)$ }{
			\If{$(T,S)$ is realizable by a minimal transversal $E$ of $Tr(\cG)$}{\Return $E$}
		}

		\ForEach{$E\in \extk(\cH)$, such that $E$ is not included in any hyperedge of $\cH$}{
			\If{$E\in Tr(\cG)$}{
				\Return $E$
			}

		}
		\Return Yes
	}
	\caption{Dualize}
\end{algorithm}
The first loop of Algorithm \ref{algo:1} tries to find a minimal transversal $E\in Tr(\cG) \setminus \extk(\cH)$, i.e. a minimal transversal of $\cG$ that realizes a $k$-trace $(T,S)\notin \tracek(\cH)$. To do so we simply go over all $k$-traces that are not in $\tracek(\cH)$ and check for each one whether it is realizable by a minimal transversal of $\cG$. There are at most $2^k\binom{|V|}{k}$ such $k$-traces and we will show in Corollary \ref{cor:traceCheck} that we can check in polynomial time whether a $k$-trace is realizable by a minimal transversal of $\cG$.

The second loop of Algorithm \ref{algo:1} tries to find a minimal transversal of $\cG$ that is $k$-compatible with $\cH$ and that does not belong to $\cH$. For this, we first generate $\extk(\cH)$ from $\cH$ and we try every possible hyperedge of $\extk(\cH)$ that is not included in a hyperedge of $\cH$. In general the size of $\extk(\cH)$ may be exponentially larger than the size of $\cH$, but we prove in Proposition \ref{prop:computeextk} that  when $\VCdim(\cH)<k$, its size is polynomial and we can compute it in polynomial time.




To check the condition on line 3, we need to test whether $\cG$ has a minimal transversal containing $T$ and excluding $S\setminus T$ which can be done in polynomial time when $T$ is of bounded size as we will see in section \ref{subsec:subtrans}. The remaining part is to compute $\extk(\cH)$. We will see in section \ref{sec:extk} that this can be done in  polynomial time  whenever $\VCdim(\cH)$ is bounded.

\begin{theorem}\label{thm:correctness}
	For any $k\in \mathbb{N}$, Algorithm \ref{algo:1} correctly returns "Yes" if $Tr(\cG) = \cH$ and a set $E\in Tr(\cG) \setminus \cH$ otherwise.
\end{theorem}
\begin{proof}
	Let us show first that if a set $E$ is returned by the algorithm, then $E\in Tr(\cG) \setminus \cH$.

	Assume first that $E$ is returned in line 4, i.e. that the condition in line 3 is true. The condition imposes that $E\in Tr(\cG)$ and since $E$ realizes a $k$-trace $(T,S)$ that does not belong to $\tracek(\cH)$, $E$ cannot belong to $\cH$ since otherwise all $k$-traces realized by $E$ would belong to $\tracek(\cH)$. So $E\in Tr(\cG) \setminus \cH$. Suppose now that $E$ is returned in line 7.
	Notice that we impose in line 5 that $E$ is contained in no hyperedge of $\cH$. In particular, $E$ cannot belong to $\cH$. Together with the condition in line 6, $E$ is returned in line 7 only if $E\in Tr(\cG)\setminus \cH$.
	So if $Tr(\cG) = \cH$ then the algorithm answer "Yes" since otherwise it would return a set $E \in Tr(\cG) \setminus \cH$.

	Assume now that the algorithm answer "Yes" and let us show that $Tr(\cH) = \cH$. Assume for contradiction that there exists $E\in Tr(\cG) \setminus \cH$. If $\tracek(E)\subseteq \tracek(\cH)$ then $E\in \extk(\cH)$ and then, $E$ would be returned in the loop in line 5. Otherwise there is a $k$-trace $(T,S)\in \tracek(E) \setminus \tracek(\cH)$. But then, $E$ would be a minimal transversal of $\cG$ which realizes a trace $(T,S) \notin \tracek(\cH)$ and such a minimal transversal of $\cG$ would be returned in line 4. Notice that it wouldn't be necessarily $E$ that would be returned, but any $E'\in Tr(\cG)$ that realizes $(T,S)$.
\end{proof}

\subsection{Finding a minimal transversal satisfying a $k$-trace}\label{subsec:subtrans}

In this section, we show that the condition in line 3 of Algorithm \ref{algo:1} can be checked in polynomial time. More precisely, given a $k$-trace $(T,S)$, we show that we can check in time $O(n|\cG|^{|T|+1})$ whether there exists a minimal transversal of $\cG$ realizing $(T,S)$. When $k$ is a constant, $T$ is of constant size and the above-mentioned complexity is polynomial. To do so, we reduce the problem to deciding whether $T$ is a sub-transversal of a subhypergraph of $\cG$. A set of vertices $T$ is a \emph{sub-transversal} of $\cH$ if there exists a minimal transversal $E\in Tr(\cH)$ such that $T\subseteq E$. In \cite{boros_dual_1998}, the authors proved that one can check in polynomial time whether a subset of vertices $T$ of bounded size is a sub-transversal.

\begin{lemma}[\cite{boros_dual_1998}]
	Given a subset of vertices $T$, one can check in time $O(n|\cH|^{|T|+1})$ whether $T$ is a sub-transversal of $\cH$.
\end{lemma}

\begin{lemma}
	Let  $(T,S)$ be a $k$-trace on $V$ and let $V':= V\setminus (S\setminus T)$. Then, there exists a minimal transversal of $\cG$ realizing $(T,S)$ if and only if $T$ is a sub-transversal of the subhypergraph $\cG'=(V',\{F\cap V'\mid F\in \cG\})$
\end{lemma}
\begin{proof}
	Observe first that $E\in Tr(\cG')$ if and only if $E$ is a minimal transversal of $\cG$ such that $E\subseteq V'$. Assume first that $T$ is a sub-transversal of $\cG'$. There exists $E\in Tr(\cG')$ such that $T\subseteq E$. Since $E$ is a minimal transversal of $\cG'$, it is also a minimal transversal of $\cG$. Now, since $E\subseteq V'$ and $V'\cap S = T$, we have $E \cap S = T$, and thus $E$ realizes $(T,S)$.
	Assume now that there exists $E\in Tr(\cG)$ such that $E$ realizes $(T,S)$. We have $E\cap S = T$, and hence $E\subseteq V'$. Therefore, $E$ is a minimal transversal of $\cG'$ containing $T$.
\end{proof}

Note that in the lemma above, $\cG'$ may contain empty hyperedges, namely when some hyperedges of $\cG$ are entirely contained in $S\setminus T$ (and thus become empty after restriction to $V'$). In that case, no minimal transversal of $\cG$ can realize $(T,S)$, which is consistent with the fact that $\cG'$ has no minimal transversals (a hypergraph with an empty hyperedge has no transversals).

Combining the two previous lemmas, we obtain the following corollary.

\begin{corollary}\label{cor:traceCheck}
	Given a $k$-trace $(T,S)$, one can check in time $O(n|\cG|^{|T|+1})$ whether there exists a minimal transversal $E$ of $\cG$ realizing $(T,S)$.
\end{corollary}

\subsection{Computing $\extk(\cH)$}\label{sec:extk}
In this section we show that when $\VCdim(\cH)<k$, $\extk(\cH)$ can be computed in time $O(n^{2k})$.


For $i\leq n$, let us denote by $V_i$ the set $\{v_1,\ldots,v_i\}$ and let $\cH_i$ be the hypergraph $(V_i,\{F\cap V_i \mid F\in \cH\})$.
We iteratively build $\extk(\cH_i)$ from $\extk(\cH_{i-1})$. For each $E\in \extk(\cH_{i-1})$, we check whether $E$ and $E\cup\{v_i\}$ violate a $k$-trace in $\tracek(\cH_{i})$.

\begin{lemma}\label{lem:initial}
	$\extk(\cH_k) = \cH_k$.
\end{lemma}
\begin{proof}
	As noted previously, for any hypergraph $\cH$, we have $\cH \subseteq \extk(\cH)$, and hence $\cH_k \subseteq \extk(\cH_k)$. Let us show that $\extk(\cH_k) \subseteq \cH_k$. Let $E\in \extk(\cH_k)$. Since $|V_k| = k$, the only $k$-trace realized by $E$ on $V_k$ is the $k$-trace $(E\cap V_k, V_k)$. Now, since $E\in \extk(\cH_k)$, there exists a hyperedge $F\in \cH_k$ which realizes the $k$-trace $(E\cap V_k, V_k)$, i.e., such that $F\cap V_k = E\cap V_k$. But since both $F\subseteq V_k$ and $E\subseteq V_k$, we have $F=E$, i.e., $E$ is a hyperedge of $\cH_k$.
\end{proof}

\begin{lemma}\label{lem:nextext}
	Let $k<i\leq n$ and let $E\in \extk(\cH_i)$. Then $E\setminus \{v_i\}\in \extk(\cH_{i-1})$.
\end{lemma}
\begin{proof}
	Let $E':= E\setminus \{v_i\}$. Let us show that $\tracek(E',V_{i-1})\subseteq \tracek(\cH_{i-1})$. Let $S\subseteq V_{i-1}$ with $|S|=k$. Since $E\in \extk(\cH_i)$, there exists $F\in \cH_{i}$ such that $E\cap S = F\cap S$. Let $F':= F\setminus \{v_i\}$. Since $v_i\notin S$, $E'\cap S = F'\cap S$. Since $F$ is a hyperedge of $\cH_i$, $F'$ is a hyperedge of $\cH_{i-1}$ and so the trace $(E'\cap S,S) = (F'\cap S,S)$ belongs to $\tracek(\cH_{i-1})$. This holds for every $k$-subset $S\subseteq V_{i-1}$, hence $E'\in \extk(\cH_{i-1})$.
\end{proof}

\begin{proposition}\label{prop:computeextk}
	If $\VCdim(\cH)<k$ then $\extk(\cH)$ can be computed in time $O(n^{2k})$.
\end{proposition}
\begin{proof}
	To compute $\extk(\cH)$, we iteratively compute $\extk(\cH_i)$ from $\extk(\cH_{i-1})$ up to $\extk(\cH_n)=\extk(\cH)$. By Lemma \ref{lem:initial}, we have $\extk(\cH_k)=\cH_k$, and thus we can start the iteration at $i=k$. Then, for each $k<i\leq n$, we compute $\extk(\cH_i)$ from $\extk(\cH_{i-1})$ in the following way:

	\noindent For each $E\in \extk(\cH_{i-1})$ Do:\\
	\indent Add $E$ to $\extk(\cH_i)$ if $\tracek(E,V_i) \subseteq \tracek(\cH_i)$\\
	\indent Add $E':=E\cup\{v_i\}$ to $\extk(\cH_i)$ if $\tracek(E',V_i) \subseteq \tracek(\cH_i)$.\\

	By Lemma \ref{lem:nextext}, each hyperedge of $\extk(\cH_i)$ is found by the above procedure. Now observe that for any $i\leq n$, the hypergraph $\cH_i$ is a subhypergraph of $\cH$ and so $\VCdim(\cH_i)\leq \VCdim(\cH)<k$. Thus, by Corollary \ref{cor:sauer}, $|\extk(\cH_i)|=O(i^k)=O(n^{k})$. Now, for a given $E\in \extk(\cH_{i-1})$, we need to check whether $E$ (resp. $E':=E\cup \{v_i\}$) belongs to $\extk(\cH_{i})$, i.e., whether $\tracek(E,V_i) \subseteq \tracek(\cH_i)$ (resp. $\tracek(E',V_i) \subseteq \tracek(\cH_i)$). We claim that this can be checked in time $O((i-1)^{k-1})=O(n^{k-1})$. Let $S\subseteq V_i$ be a $k$-subset such that $v_i\notin S$. Since $S\subseteq V_{i-1}$ and since $E\in \extk(\cH_{i-1})$, there exists $F\in \cH_{i-1}$ such that $E\cap S=F\cap S$. Now there exists $F'\in \cH_i$ such that $F=F'\setminus\{v_i\}$ and since $v_i\notin S$, we have $F'\cap S = F\cap S = E\cap S = E'\cap S$. Thus, the $k$-trace $(E\cap S,S)\in \tracek(\cH_i)$ (resp. $(E'\cap S,S)\in \tracek(\cH_i)$). So if $E$ (resp. $E'$) realizes a trace $(E\cap S,S)\notin \tracek(\cH_i)$ (resp. $(E'\cap S,S)\notin \tracek(\cH_i)$), then $S$ must contain $v_i$. Therefore, to check whether $\tracek(E,V_i) \subseteq \tracek(\cH_i)$ (resp. $\tracek(E',V_i) \subseteq \tracek(\cH_i)$), we only have to check, for each of the $O((i-1)^{k-1})$ subsets $S$ of $V_i$ of size $k$ that contain $v_i$, whether $(E\cap S,S)\in \tracek(\cH_i)$.

	Hence, $\extk(\cH_i)$ is computed in $O(n^{2k-1})$ since we check for each of the $O(n^{k})$ hyperedges $E$ of $\extk(\cH_{i-1})$ whether $E$ and $E'$ belong to $\extk(\cH_i)$, which can be done in $O(n^{k-1})$. Thus, in total, $\extk(\cH)$ can be computed in time $O(n\times n^{2k-1})=O(n^{2k})$.


\end{proof}

We are now ready to prove Theorem \ref{thm:main}

\mainthm*
\begin{proof}
	From Theorem \ref{thm:correctness} Algorithm \ref{algo:1} returns "Yes" if and only if $Tr(\cG) = \cH$ and $E\in Tr(\cG)\setminus  \cH$ otherwise. By the duality property, the algorithm returns "Yes" if and only if $Tr(\cH) = \cG$. Assume now that the algorithm returns an element $E\in Tr(\cG)\setminus \cH$. Since we assumed that all hyperedges of $\cH$ are transversals of $\cG$, $E$ does not contain any heyperedge of $\cH$ and thus $V\setminus E$ is a transversal of $\cH$. Observe furthermore that $V\setminus E$ contains no hyperedge of $\cG$ since otherwise $E$ would not be a transversal of $\cG$. So to output a minimal transversal $E' \in Tr(\cH) \setminus \cG$ it is enough to compute any minimal transversal of $\cH$ contained in $V\setminus E$ which can be done in $O(n|\cH|)$.
	Let us show now that Algorithm \ref{algo:1} runs with the announced complexity. The for loop in line 2 goes over all $k$-traces $(T,S)\notin \tracek(\cH)$. There is at most $O(2^{k}n^{k})$ such traces and for each $k$-trace the condition in line 3 can be checked in time $O(n|\cG|^{k+1})$ by Corollary \ref{cor:traceCheck}. So the first for loop take in total $O(2^k(n|\cG|)^{k+1})$.

	For the second loop, by Lemma \ref{prop:computeextk}, we can compute $\extk(\cH)$ in time $O(n^{2k})$. Since for each $E\in \extk(\cH)$ one can check whether $E\in Tr(\cG)$ in time $O(n|\cG|)$ the total time taken by the for loop in line 2 is $O(n^{2k+1}|\cG|)$.

\end{proof}

\subsection{Consequences}

The first immediate consequence of Theorem \ref{thm:main} is obtained when the VC-dimension is bounded, i.e. when $k$ is a constant.

\begin{corollary}\label{cor:vc-bounded}
	\Thyp problem is solvable in polynomial time in hypergraph classes of bounded VC-dimension.
\end{corollary}

Even if several quasi-polynomial time algorithms are already known for general hypergraphs, we observe Algorithm \ref{algo:1} runs also in quasi-polynomial time.

\begin{corollary}
	Algorithm \ref{algo:1} runs in time $(n|\cG|)^{O(\log(\cH))}$.
\end{corollary}
\begin{proof}
	Notice that if a set of size $k$ is shattered by $\cH$, then $\cH$ contains at least $2^{k}$ hyperedges, and so we have $\VCdim(\cH)= O(\log(\cH))$. So by Theorem \ref{thm:main}, we obtain the result.
\end{proof}

Another important property of Algorithm \ref{algo:1} is that it runs in polynomial time on hypergraphs of bounded conformality. The conformality of hypergraph was introduced in \cite{berge_hypergraphs_1989}, a hypergraph $\cH$ is said to be $k$-conformal if any minimal subset of vertices that is not included in any hyperedge of $\cH$ is of size at most $k$. In \cite{khachiyan_dualization_2007}, the authors prove that \Thyp is solvable in polynomial time in $k$-conformal hypergraphs. Although $k$-conformal hypergraphs can have arbitrarily large VC-dimension, Algorithm \ref{algo:1} runs in polynomial time for hypergraphs of bounded conformality. In the proof of Theorem \ref{thm:main} the VC-dimension is used to bound the size $\extk(\cH)$ and to be able to compute it in polynomial time. The following proposition asserts that if $\cH$ is $k$-conformal, then any hyperedge of $\extk(\cH)$ is a subset of a hyperedge of $\cH$. Hence, even though the size of $\extk(\cH)$ may be exponential in $\cH$, we don't need to compute it since the for loop in line 5 of Algorithm \ref{algo:1} only runs through hyperedges of $\extk(\cH)$ not included in any hyperedge of $\cH$.

\begin{proposition}\label{prop:conformal}
	Let $\cH$ be a $k$-conformal hypergraph. If $E\in \extk(\cH)$, there exists $F \in \cH$ such that $E\subseteq F$.
\end{proposition}
\begin{proof}
	Assume that there exists $E\in \extk(\cH)$ such that $E$ is not included in any hyperedge of $\cH$. Let $E' \subseteq E$ be a minimal subset that is contained in no hyperedge of $\cH$.
	$E'$ is well defined since $E$ itself is contained in no hyperedge of $\cH$. By definition of the conformality, we have $|E'|\leq k$.
	Let $S$ be any superset of $E'$ of size $k$. Then the $k$-trace $(E\cap S,S)$ is realized by $E$ but does not belong to $\trace_k(\cH)$ since $E'\subseteq E\cap S$ while no hyperedge of $\cH$ contains $E'$. So $E$ realizes a $k$-trace which does not belongs to $\trace_k(\cH)$ which is in contradiction with $E\in \extk(\cH)$.

\end{proof}

\begin{corollary}\label{cor:conform}
	Algorithm \ref{algo:1} runs in time $O(2^k(n|\cG|)^{k+1})$ on $k$-conformal hypergraphs
\end{corollary}
\begin{proof}
	By  Proposition \ref{prop:conformal} the second for loop is empty, and only the for loop in line 2 has to be taken into account. As already shown in the proof of Theorem \ref{thm:main} it runs in time $O(2^k(n|\cG|)^{k+1})$.
\end{proof}

\section{Hypergraph classes closed under subhypergraphs}\label{sec:closed}

As already mentioned in the introduction, a consequence of Theorem \ref{thm:main} or more precisely of corollary \ref{cor:vc-bounded} is that \Thyp is solvable in polynomial time in any proper class of hypergraph closed under partial subhypergraph. Indeed, any class of hypergraphs $\mathscr{H}$ that forbids a hypergraph $\cH$ as partial subhypergraph is of bounded dimension and thus, Algorithm \ref{algo:1} can be used to solve the \Thyp problem. We also show in this section that such a result is hopeless for the other notions of subhypergraph except if one can solve the \Thyp problem in general hypergraphs in polynomial time.

\subsection{Classes closed under partial subhypergraphs.}

\begin{lemma}\label{lem:universalvc}
	If $\VCdim(\cH) = k$, then $\cH$ contains all hypergraphs on at most $k$ vertices as partial subhypergraph.
\end{lemma}
\begin{proof}
	Let $\cH'=(V':=\{v_1,...,v_{|V'|}\},\cF')$ be a hypergraph with $|V'|\leq k$. Since $\VCdim(\cH)=k$, $\cH$ has a shattered set of vertices $T:=\{t_1,...,t_{|V'|}\}$ of size $|V'|$. Now since $T$ is shattered, for every $F:=\{v_{i_1},...,v_{i_\ell}\} \in \cF'$ there exists $f(F)\in \cH$ such that $f(F) \cap T = \{t_{i_1},...,t_{i_\ell}\}$ and the partial subhypergraph $(T,\{f(F)\mid F\in \cF'\})$ of $\cH$ is a copy of $\cH'$.
\end{proof}

\begin{proposition}\label{prop:partial-vc}
	Any proper class of hypergraphs $\mathscr{H}$ which is closed under partial subhypergraphs has a bounded VC-dimension.
\end{proposition}
\begin{proof}
	Let $k$ be the minimum number of vertices of a hypergraph that is not in $\mathscr{H}$. Such a hypergraph exists since $\mathscr{H}$ is a proper class of hypergraphs (i.e., it does not contain all hypergraphs). Let us show that the VC-dimension of $\mathscr{H}$  is bounded by $k$. Indeed let $\cH\in \mathscr{H}$. Since $\mathscr{H}$ is closed under partial subhypergraph and since $\cH'\notin \mathscr{H}$, $\cH$ does not contain $\cH'$ as partial subhypergraph. Hence, by Lemma \ref{lem:universalvc}, $\VCdim(\cH)<k$.
\end{proof}

\closedHyper*

\subsection{Other subhypergraphs types}

Several types of subhypergraphs have been defined in the literature namely, \emph{partial hypergraphs}, \emph{subhypergraphs}, \emph{partial subhypergraphs}, \emph{edge-induced subhypergraphs} and \emph{restrictions}.
Even if the concepts of partial hypergraphs, subhypergraphs and partial subhypergraphs have been already defined, we recall them for comparison to the other notions. For a hypergraph $\cH=(V,\cF)$,
\begin{itemize}
	\item a \emph{partial hypergraph} of $\cH$ is a hypergraph $\cH'=(V,\cF')$, for $\cF'\subseteq \cF$;
	\item   a \emph{subhypergraph} of $\cH$ is a hypergraph $\cH'=(V',\{F\cap V' \mid F\in \cF, F\cap V'\neq \emptyset\})$, for $V'\subseteq V$;
	\item a \emph{partial subhypergraph} of $\cH$ is a hypergraph $\cH'=(V',\{F\cap V' \mid F\in \cF'\})$, for $V'\subseteq V$ and $\cF'\subseteq \cF$;
	\item an \emph{edge-induced subhypergraph} of $\cH$ is a hypergraph $\cH'=(\bigcup\limits_{F\in \cF'}F,\cF')$, for $\cF'\subseteq \cF$;
	\item a \emph{restriction} of $\cH$ is a hypergraph $\cH'=(V',\{F\in \cF \mid F \subseteq V' \})$, for $V'\subseteq V$.
\end{itemize}

We show that a result similar to the one obtained in Corollary \ref{cor:closed} for the other notions of subhypergraphs would imply a polynomial time algorithm for the general problem.
More precisely we focus on the class of hypergraphs that forbids a specific hypergraph $\cF$ as subhypergraph. Given a hypergraph $\cF$, we denote by $\mathscr{H}_\cF^{\textsc{S}}$, $\mathscr{H}_\cF^{\textsc{Es}}$, $\mathscr{H}_\cF^{\textsc{R}}$, the set of hypergraphs that forbid $\cF$ respectively as subgraph, edge-induced subgraph and restriction.
For the case of edge-induced subhypergraphs, for  any hypergraph $\cF$, we prove that the \Thyp problem in the class $\mathscr{H}_\cF^{\textsc{Es}}$ is as hard as in the general case.

For the case of restrictions and induced subhypergraphs, the difficulty of \Thyp in the classes $\mathscr{H}_\cF^{\textsc{R}}$ and  $\mathscr{H}_\cF^{\textsc{S}}$ depends on the hypergraph $\cF$. For example, if $\cF$ is the complete hypergraph on $k$ vertices (the hypergraph that contains all subsets as hyperedges), then the class $\mathscr{H}_\cF^{\textsc{S}}$ is precisely the class of hypergraphs of VC-dimension strictly smaller than $k$ and by Theorem \ref{thm:main} the problem is polynomial. However, if the hypergraph $\cF$ is sparse enough, we prove that the \Thyp problem restricted to the class $\mathscr{H}_\cF^{\textsc{S}}$ is as hard as in general hypergraphs.

For the case of restrictions, if $\cF$ is the empty hypergraph on $k$ vertices (the hypergraph without any hyperedge), then the class $\mathscr{H}_\cF^{\textsc{R}}$ is the class of hypergraph for which the maximum independent set is strictly smaller than $k$. Since maximal independent sets are the complements of minimal transversals, one can enumerate all subsets of size smaller than $k$ and check whether the complement is a minimal transversal. Therefore, if $\cF$ is the empty hypergraph, then \Thyp problem  in $\mathscr{H}_\cF^{\textsc{R}}$ is polynomial. In any other case, we prove that \Thyp problem restricted to $\mathscr{H}_\cF^{\textsc{R}}$ is as hard as the general problem.

In order to prove the results of this section, we will use two different reductions. Given a hypergraph $\cH$ and an integer $k$, we define $\hat{\cH}^{k}$ to be the hypergraph obtained from $\cH$ by adding  $k$ new vertices to every hyperedges, and $\hat{\cH}_{k}$ to be the hypergraph obtained from $\cH$ by adding $\binom{n}{k}$ new hyperedges of size $k+1$ corresponding to all $k$-subsets of vertices of $\cH$ plus a new vertex $x$.
\begin{itemize}
	\item  $\hat{\cH}^{k} = (V(\cH)\cup \{x_1,...,x_k\}, \{E\cup \{x_1,...,x_k\} \mid E \in \cE(\cH)\})$
	\item $\hat{\cH}_{k} = (V(\cH)\cup \{x\}, \cE(\cH) \cup \{X\cup \{x\} \mid X \subseteq V(\cH), |X| \leq k\})$
\end{itemize}

We now show how the enumeration of minimal transversal of $\cH$ can be reduced to the enumeration of minimal transversal of $\hat{\cH}^{k}$ and $\hat{\cH}_{k}$.

\begin{proposition}\label{prop:hatup}
	Let $k$ be an integer and $\cH$ be a hypergraph,  then
	$Tr(\hat{\cH}^{k}) = Tr(\cH) \cup \{\{x\} \mid x\in V(\hat{\cH}^{k})\setminus V(\cH)\}$.
\end{proposition}
\begin{proof}
	Let $T\in Tr(\hat{\cH}^{k})$. Then either $T\subseteq V(\cH)$ and $T$ is also a minimal transversal of $\cH$ or $T$ contains at least one vertex  $x\in V(\hat{\cH}^{k}) \setminus V(\cH)$. Since $x$ belongs to every hyperedges of $\hat{\cH}^{k}$ $\{x\}$ is a transversal and by minimality of $T$ we have  $T=\{x\}$. So $Tr(\hat{\cH}^{k}) \subseteq Tr(\cH) \cup \{\{x\} \mid x\in V(\hat{\cH}^{k})\setminus V(\cH)\}$. Now since every minimal transversal of $\cH$ is also a minimal transversal of $\hat{\cH}^{k}$ and any $x\in V(\hat{\cH}^{k}) \setminus V(\cH)$ belongs to every hyperedge of  $\hat{\cH}^{k}$, we have $Tr(\hat{\cH}^{k}) \supseteq Tr(\cH) \cup \{\{x\} \mid x\in V(\hat{\cH}^{k})\setminus V(\cH)\}$.
\end{proof}

\begin{proposition}\label{prop:hatdown}
	Let $k$ be an integer, $\cH$ be a hypergraph, then
	$Tr(\hat{\cH}_{k}) = \{T\cup\{x\} \mid T \in Tr(\cH), |T| < n-k+1\} \cup \cJ$ where $\cJ$ only contains subsets of size larger than $n-k+1$.
\end{proposition}
\begin{proof}
	Let $T\in Tr(\hat{\cH}_{k})$. Assume first that the only vertex $x\in V(\hat{\cH}_{k}) \setminus V(\cH)$ belongs to $T$. Then $T\setminus \{x\}$ must minimally cover the hyperedges that don't contain $x$ i.e. $T\setminus \{x\}$ must be a minimal transversal of $\cH$ since the hyperedges that contain $x$ are exactly $\cE(\hat{\cH}_{k}) \setminus \cE(\cH)$. Assume now that $x\notin T$ and assume that $|T|< n-k +1$. Let $X \subseteq V(\cH) \setminus T$ be a subset of vertices included in the complement of $T$ such that $|X|=k$, then $X\cup\{x\}$ is a hyperedge of $\hat{\cH}_{k}$ whose intersection with $T$ would be empty contradicting the fact that $T$ is a transversal of $\hat{\cH}_{k}$. So if $T\in Tr(\hat{\cH}_{k})$ then either $x\in T$ and  $T\setminus \{x\}$ is a minimal transversal of $\cH$ or $|T| \geq n-k+1$. Assume now that $T\in Tr(\cH)$ with $|T|< n-k+1$ and let us show that $T\cup \{x\} \in Tr(\hat{\cH}_{k})$. Since $T$ is a minimal transversal of $\cH$ and since $x$ belongs to every hyperedge of $\cE(\hat{\cH}_{k}) \setminus \cE(\cH)$, $T\cup\{x\}$ is a transversal of $\hat{\cH}_{k}$. Let us show that it is minimal. Observe first that since $|T|< n-k+1$, $T$ cannot be a transversal since its complement contains at least one $k$-subset $X$ of $V(\cH)$ and the hyperedge  $X\cup \{x\}$ would be avoided by $T$. Now if there exists $T'\subset T$ such that $T'\cup\{x\}\in Tr(\hat{\cH}_{k})$, since $x$ does not belong to any original hyperedge of $\cH$, $T'$ would be a transversal of $\cH$ contradicting the minimality of $T$.
\end{proof}

\begin{lemma}\label{lem:sub_hypergraph}
	Let $\cF$ be a fixed  hypergraph on $k$ vertices,  and let $\cH$ be a hypergraph on at least $k$ vertices, then :
	\begin{enumerate}
		\item  $\hat{\cH}^{k+1}$ is edge-induced $\cF$-free.
		\item  If $\cF$ has at least one hyperedge, $\hat{\cH}^{k+1}$ is restriction $\cF$-free.
		\item  If there exists $\ell \leq k$ such that $\cF$ has strictly less than $\binom{k-1}{\ell-1}$ hyperedges of size $\ell$, then $\hat{\cH}_{\ell}$ is subhypergraph $\cF$-free.
	\end{enumerate}
\end{lemma}

\begin{proof}
	1) Since every hyperedges of $\hat{\cH}^{k+1}$ contains the set $X:=\{x_1,...,x_{k+1}\}$ of additional vertices, any edge-induced subgraph of $\hat{\cH}^{k+1}$ contains at least these $k+1$ vertices. Since $\cF$ has only $k$ vertices, no edge-induced subgraph of $\hat{\cH}^{k+1}$ can be isomorphic to $\cF$.

	2) Let $U\subseteq V(\cH)$. Since $X$  is included in every hyperedges of $\hat{\cH}^{k+1}$, if $X\not\subseteq U$, the restriction of  $\hat{\cH}^{k+1}$ to $U$ contains no hyperedge. So any non empty restriction of $\hat{\cH}^{k+1}$ must contain $X$ and so contains at least $k+1$ vertices. Since $\cF$ has exactly $k$ vertices, $\hat{\cH}^{k+1}$ is restriction $\cF$-free.

	3) Assume that there exists $\ell\leq k$ such that $\cF$ has strictly less than $\binom{k-1}{\ell-1}$ hyperedges of size $\ell$ and let $U\subseteq V(\cH)$ be a subset of size $k$. Let $x$ be the extra vertex added in $V(\hat{\cH}_{\ell+1})$ not present in $V(\cH)$.  If $U$ does not contain $x$, any subset $E$ of $U$ of size $\ell$ will  forms a hyperedge with $x$ . Since $x$ does not belong to $E$, $E$ is an hyperedge of the subhypergraph induced by $U$. So the subhypergraph induced by $U$ contains  $\binom{k}{\ell}\geq\binom{k-1}{\ell-1} $ hyperedges of size $\ell$ and cannot be isomorphic to $\cF$. Assume now that $x\in U$. Since we assumed that $\cH$ has at least $k$ vertices,  $\hat{\cH}_{\ell}$ has at least $k+1$ vertices  and there exists a vertex $y\notin U$. Now for any subset $E$ of $U\setminus \{x\}$, of size $\ell -1$, there exists a hyperedge $E\cup \{x,y\}$ in $\hat{\cH}_{\ell}$ and so the subhypergraph induced by $U$ has $E\cup\{x\}$ as hyperedge. So the subhypergraph induced by $U$ contains at least $\binom{k-1}{\ell-1}$ hyperedges of size $\ell$ and cannot be isomorphic to $\cF$.
\end{proof}


\begin{theorem}\label{thm:sub_hypergraph}
	Let $\cF$ be a hypergraph, then the restriction of the \Thyp problem to the following classes of hypergraph is as hard as the \Thyp problem in general hypergraphs:
	\begin{itemize}
		\item  $\mathscr{H}_\cF^{\textsc{Es}}$
		\item  $\mathscr{H}_\cF^{\textsc{R}}$ if $\cF$ is different from the empty hypergraph.
		\item  $\mathscr{H}_\cF^{\textsc{S}}$ if there exists $\ell \leq |V(\cF)|$ such that $\cF$ has less than $\binom{|V(\cF)| - 1}{\ell-1}$ hyperedges of size $\ell$.
	\end{itemize}
\end{theorem}
\begin{proof}
	By lemma \ref{lem:sub_hypergraph} it is sufficient to prove that for a given integer $k$, one can reduce in polynomial time the enumeration of minimal transversal of $\cH$ to the enumeration of minimal transversal of $\hat{\cH}^{k}$ and $\hat{\cH}_{k}$.

	Assume that one can enumerate the minimal transversal of  $\hat{\cH}^{k}$ in $O((|\hat{\cH}^{k}| + |Tr(\hat{\cH}^{k})|)^{\ell})$. By lemma \ref{prop:hatup},  $Tr(\hat{\cH}^{k}) = Tr(\cH) \cup \{\{x\} \mid x\in V(\hat{\cH}^{k})\setminus V(\cH)\}$. So one can simply check for each $T\in Tr(\hat{\cH}^{k})$ whether $T$ is a minimal transversal of $\cH$. Since $|\hat{\cH}^{k}| = |\cH|$ and  $|Tr(\hat{\cH}^{k})| = |Tr(\cH)| + k$, the total running time is $O((|\cH|+|Tr(\cH)|)^{\ell})$.

	Assume now that one can enumerate the minimal transversal of  $\hat{\cH}_{k}$ in $O((|\hat{\cH}^{k}| + |Tr(\hat{\cH}^{k})|)^{\ell})$.  By lemma \ref{prop:hatdown}, $Tr(\hat{\cH}_{k}) = \{T\cup\{x\} \mid T \in Tr(\cH), |T|<n-k+1\} \cup \cJ$ where $\cJ$ only contains subsets of size larger than $n-k+1$. One can start by enumerating all minimal transversal of $\cH$ of size larger than $n-k+1$ in a brute-force way by checking all the $O(n^{k-1})$ such subsets. Then for every $T\in Tr(\hat{\cH}_{k})$ with $|T|\leq n-k+1$ that contains $x$, output $T\setminus \{x\}$ (where $x$ is the additional vertex of $V(\hat{\cH}^{k})\setminus V(\cH)$). Since $|\hat{\cH}_{k}| = O(|\cH| + n^k)$ and $|Tr(\hat{\cH}_{k})| = O(|Tr(\cH)| + n^k)$, the total running time is $O(|\cH| + |Tr(\cH)|)^{\ell k}$.

\end{proof}

\section{Discussions}
Although the existence of a polynomial-time algorithm for \Thyp remains a widely open problem, we show that \Thyp can be solved in polynomial time whenever one of the two input hypergraphs has bounded VC-dimension. This result generalizes many existing tractable cases and yields a new quasi-polynomial-time algorithm for general hypergraphs. 

A natural question is whether \Thyp, parameterized by the minimum VC-dimension of the two input hypergraphs, is fixed-parameter tractable (FPT). From the viewpoint of parameterized complexity, our result shows that the problem is in XP (i.e., solvable in time $N^{f(k)}$ for some function $f$, where $N$ is the total input size), but the existence of an FPT algorithm remains open.

\noindent\textbf{Open problem 1:} Given $\cH$ and $\cG$ with $\VCdim(\cH)<k$, can we decide in time $f(k)\cdot p(|\cH| + |\cG|)$ whether $\cH$ and $\cG$ are dual, where $f$ is a computable function and $p$ is a polynomial?

As a weaker question, can one remove the exponential dependence on $k$ from the contribution of $|\cG|$ in the running time? This would make sense from the perspective of the enumeration problem (\Tenum), since the exponential dependence on $k$ would then be carried only by the input size. Since, by Lemma \ref{lem:sauer}, the number of hyperedges of $\cH$ is polynomially bounded in $n$, this amounts to asking whether the exponential dependence on $k$ can be confined to the number of vertices.

\noindent\textbf{Open problem 2:} Given $\cH$ and $\cG$ with $\VCdim(\cH)<k$, can we decide in time $n^{f(k)}\cdot p(|\cH| + |\cG|)$ whether $\cH$ and $\cG$ are dual, where $f$ is a computable function, $p$ is a polynomial, and $n$ is the number of vertices of $\cH$ (equivalently, of $\cG$)?

Finally, for which hypergraphs $\cF$ can \Tenum, restricted to the class $\mathscr{H}_\cF^{\textsc{S}}$ of hypergraphs forbidding $\cF$ as an induced subhypergraph, be solved in polynomial time? By Theorem \ref{thm:sub_hypergraph}, a positive answer for sparse hypergraphs $\cF$ would imply a polynomial-time algorithm in the general case, but the question remains open for dense hypergraphs.

\bibliographystyle{plain}
\bibliography{bibVC}

\end{document}